\documentclass[12pt,a4paper]{article}%
\usepackage{amsmath, amsfonts, amsthm,color, latexsym, amssymb}
\usepackage{amscd}
\usepackage{amsmath}
\usepackage{amsfonts}
\usepackage{amssymb}
\usepackage{graphicx}%
\usepackage[active]{srcltx}
\providecommand{\U}[1]{\protect\rule{.1in}{.1in}}
\providecommand{\U}[1]{\protect\rule{.1in}{.1in}}
\providecommand{\U}[1]{\protect\rule{.1in}{.1in}}
\newtheorem{theorem}{Theorem}[section]
\newtheorem{proposition}[theorem]{Proposition}
\newtheorem{corollary}[theorem]{Corollary}
\newtheorem{example}[theorem]{Example}

\newtheorem{lemma}[theorem]{Lemma}
\newtheorem{definition}[theorem]{Definition}

\begin{document}

\title{Mixing convolution operators on spaces of entire functions}
\author{V. V.
F\'{a}varo\thanks{Supported by FAPESP Grant 2014/50536-7; FAPEMIG Grant PPM-00086-14; and CNPq Grants 482515/2013-9, 307517/2014-4.
}\,\,, J.
Mujica}
\date{}
\maketitle

\vspace{-0.8 cm}

\begin{center}
\it Dedicated to the memory of Jorge Alberto Barroso (1928-2015)
\end{center}

\vspace{0.2 cm}

\begin{abstract}
We show that if $E$ is an arbitrary $(DFN)$-space, then every nontrivial convolution operator on the Fr\'echet nuclear space $\mathcal{H}(E)$ is mixing, in particular hypercyclic. More generally we obtain the same conclusion when $E=F^{\prime}_c,$ where $F$ is a separable Fr\'echet space with the approximation property. On the opposite direction we show that a translation operator on the space $\mathcal{H}(\mathbb{C}^{\mathbb{N}}) $ is never hypercyclic.
\end{abstract}

\section{Introduction}
If $X$ is a topological vector space, then a continuous linear operator $T\colon X\rightarrow X$ is said to be \emph{hypercyclic}
if its \emph{orbit} $\{x,T(x),T^{2}(x),\ldots\}$ is dense in $X$ for some $x\in X$. In
this case, $x$ is said to be a \emph{hypercyclic vector for $T$}. There are several directions and ramifications of the study of hypercyclic operators. References \cite{bay, goswinBAMS} can provide an overview of the theory.

In this work we are particularly interested in the hypercyclicity of operators on spaces of entire functions.
Godefroy and Shapiro \cite{godefroy} proved that every nontrivial convolution operator on the space $\mathcal{H}(\mathbb{C}^n)$ of entire functions of several complex variables is hypercyclic. By a \emph{nontrivial convolution operator} we mean a convolution operator which is not a scalar multiple of the identity. 
This result generalizes classical results of Birkhoff \cite{birkhoff}
 and MacLane \cite{maclane} on hypercyclicity of translation and differentiation ope\-ra\-tors, respectively, on the space $\mathcal{H}(\mathbb{C})$ of entire functions of one complex variable.
Also, several results on the hypercyclicity of operators on spaces of entire functions on infinite dimensional Banach spaces have appeared in the last few decades (see,
e.g., \cite{aronbes, BBFJ, bes2012, CDSjmaa, GS,
goswinBAMS, MPS, peterssonjmaa}).
In \cite{BBFJ}, the authors proved a general result on hypercyclicity of convolution operators on the space $\mathcal{H}_{\Theta b}(E)$ of entire functions of $\Theta$-bounded type on a complex Banach space $E$.

If $X$ is a topological vector space, then a continuous linear operator $T\colon X\rightarrow X$ is said to be \emph{mixing} if for any two non-empty open sets $U,V\subset X,$ there is $n_0\in \mathbb{N}$ such that $T^n(U)\cap V\neq\emptyset$, for all $n\geq n_0.$ Actually it is known that if $X$ is a Fr\'echet space, then a continuous linear operator $T\colon X\rightarrow X$ is mixing if and only if it is \emph{hereditarily hypercyclic}, that is, for any strictly increasing sequence $(n_j)\subset\mathbb{N}$ there exists $x\in X$ such that the sequence $(T^{n_j}(x))$ is dense in $X$ (this is proved in \cite{grivaux} in the case of Banach spaces, but the proof works equally well in the case of Fr\'echet spaces).

In this work we show that if $E$ is an arbitrary $(DFN)$-space, then every nontrivial convolution operator on the Fr\'echet space $\mathcal{H}(E)$ is mixing. More generally we obtain the same conclusion when $E=F^{\prime}_c,$ where $F$ is a separable Fr\'echet space with the approximation property. It is clear that the first result follows from the second one, but we have preferred to present both results separately to illustrate the usefulness of different holomorphy types. Besides that we first obtained the result in the case of $(DFN)$-spaces, whose proof is simpler, and later on we extended the result to the case of $(DFC)$-spaces. 


Our proofs combine results for the spaces  $\mathcal{H}_{\Theta b}(E)$ (with $E$ a normed space) explored in \cite{BBFJ, favaro_jatoba} with a factorization method introduced by Co\-lom\-beau and Matos \cite{colombeau_matos} for the space $\mathcal{H}_{uN b}(E)$ (with $E$ a locally convex space), whose idea can be adapted for any holomorphy type instead of the nuclear type. Besides, our proofs rest on a classical hypercyclicity criterion, first obtained by Kitai \cite{kitai}, but never published, rediscovered by Gethner and Shapiro \cite{GS} and later on sharpened by Costakis and Sambarino \cite{costakis}.  

We finish the paper showing that we can not expect hypercyclicity of convolution operators when $E$ is an arbitrary locally convex space. In contrast with the fact that every nontrivial convolution operator on $\mathcal{H}(\mathbb{C}^n)$ is hypercyclic, we prove that every translation operator on  $\mathcal{H}(\mathbb{C}^{\mathbb{N}})$ is not hypercyclic, considering on $\mathcal{H}(\mathbb{C}^{\mathbb{N}})$ any of the three usual topologies: the compact open topology $\tau_0$, the Nachbin ported topology $\tau_{\omega}$, or the bornological topology $\tau_\delta.$ 

Throughout this paper $\mathbb{N}$ denotes the set of positive integers and
$\mathbb{N}_{0}$ denotes the set $\mathbb{N}\cup\{0\}$. All the locally convex spaces are assumed to be complex and Hausdorff. By $\Delta$ we mean the open unit disk in the complex field $\mathbb{C}$. If $E$ is a locally convex space, then $E_b^{\prime}$ (resp. $E_c^{\prime}$) denotes the dual $E^\prime$ of $E$ with the topology of bounded convergence (resp. compact convergence). If $E$ and $F$ are normed spaces, with $F$ complete, then the Banach space of all continuous $m$-homogeneous
polynomials from $E$ into $F$ endowed with its usual sup norm is denoted by $\mathcal{P}(^{m}E;F)$. The subspace of $\mathcal{P}(^{m}E;F)$ of all polynomials of finite type is represented by $\mathcal{P}_f(^{m}E;F)$. For $E$ and $F$ locally convex spaces, with $F$ complete,  $\mathcal{H}(E;F)$ denotes the vector space of all holomorphic mappings from $E$ into $F$.
In all these cases, when $F = \mathbb{C}$ we write $\mathcal{P}(^{m}E)$, $\mathcal{P}_f(^{m}E)$  and $\mathcal{H}(E)$
instead of $\mathcal{P}(^{m}E;\mathbb{C})$, $\mathcal{P}_f(^{m}E;\mathbb{C})$  and $\mathcal{H}(E;\mathbb{C})$, respectively. For the general theory of homogeneous polynomials and holomorphic functions we refer to Dineen \cite{Di2} and Mujica \cite{mujica}.

Finally, $cs(E)$ denotes the set of all continuous seminorms on the locally convex space $E$. If $p \in cs(E)$, 
then $E_{p}$ denotes the normed space $(E,p)/p^{-1}(0)$, and $\pi_{p}$ denotes the canonical 
surjective mapping $\pi_{p}\colon E \rightarrow E_{p}$. We say that $D \subset cs(E)$ is a 
\textit{fundamental family} if $D$ is a directed subset of $cs(E)$ and the seminorms $p \in D$ 
generate the topology of $E$. 

\section{Preliminaires}

In this section we recall the concepts and results about holomorphic functions on normed spaces that we need and we introduce some analogue concepts for holomorphic functions defined on locally convex spaces. It is important to say that all definitions of this section and all results of \cite{BBFJ} and \cite{favaro_jatoba} that we will use during the paper were originally stated for $E$ and $F$ Banach spaces,  with exception of Definitions \ref{def_holomorfas_elc}, \ref{def-borel} and \ref{def-convolution} that are introduced in this paper for the first time. But it is clear that they are still valid if we consider $E$ only normed.

\begin{definition}\rm Let $E$ and $F$ be normed spaces, with $F$ complete, and $U$ be an open subset of $E$. A mapping $f \colon U \longrightarrow F$ is said to be {\it holomorphic on $U$} if for every $a \in U$ there exists a sequence $(P_m)_{m=0}^\infty$, where each $P_m \in {\cal P}(^mE;F)$ ($\mathcal{P}(^{0}E;F)=F$), such that $f(x) = \sum_{m=0}^\infty P_m(x-a) $ uniformly on some open ball with center $a$. The $m$-homogeneous polynomial $m!P_m$ is called the {\it $m$-th derivative of $f$ at $a$} and is denoted by $\hat d^m f(a)$. In particular, if $P \in {\cal P}(^mE;F)$, $a \in E$ and $k \in \{0, 1, \ldots, m\}$, then
$$\hat d^k P(a)(x) = \frac{m!}{(m-k)!}\check P(\underbrace{x,\ldots,x}_{k\, \rm{times}},a,\ldots,a)  $$
for every $x \in E$, where $\check P$ denotes the unique symmetric $m$-linear mapping associated to $P$. 
\end{definition}

\begin{definition}\rm (Nachbin \cite{nachbin2})
\label{Definition tipo holomorfia} Let $E$ and $F$ be normed spaces, with $F$ complete. A \emph{holomorphy type} $\Theta$
from $E$ to $F$ is a sequence of Banach spaces $(\mathcal{P}_{\Theta}%
(^{m}E;F))_{m=0}^{\infty}$, the norm on each of them being denoted by
$\|\cdot\|_{\Theta}$, such that the following conditions hold true:

\item[$(1)$] Each $\mathcal{P}_{\Theta}(^{m}E;F)$ is a linear
subspace of $\mathcal{P}(^{m}E;F)$.

\item[$(2)$] $\mathcal{P}_{\Theta}(^{0}E;F)$ coincides with
$\mathcal{P}(^{0}E;F)=F$ as a normed vector space.

\item[$(3)$] There is a real number $\sigma\geq1$ for which the
following is true: given any $k\in\mathbb{N}_{0}$, $m\in\mathbb{N}_{0}$,
$k\leq m $, $a\in E$ and $P\in\mathcal{P}_{\Theta}(^{m}E;F)$, we have
\[
\hat{d}^{k}P(a)\in\mathcal{P}_{\Theta}(^{k}E;F) {\rm ~~and}
\]
\[
\left\|  \frac{1}{k!}\hat{d}^{k}P(a)\right\|  _{\Theta}\leq\sigma
^{m}\|P\|_{\Theta}\|a\|^{m-k}.
\]
\end{definition}
\noindent A holomorphy type from $E$ to $F$ shall be denoted by either $\Theta$ or $( \mathcal{P}_{\Theta}(^{m}E;F))_{m=0}^\infty$.


\begin{definition}\rm
\label{Definition f holomorfia} Let $(\mathcal{P}_{\Theta}%
(^{m}E;F))_{m=0}^{\infty}$ be a holomorphy type from the normed space $E$ to the Banach space $F$. A given
$f\in\mathcal{H}(E;F)$ is said to be of \emph{$\Theta$-bounded type} if

\item[$(1)$] $\frac{1}{m!}\hat{d}^{m}f(0)\in\mathcal{P}_{\Theta}%
(^{m}E;F)$ for all $m\in\mathbb{N}_{0}$,

\item[$(2)$] $\displaystyle\lim_{m \rightarrow\infty} \left(  \frac{1}{m!}%
\|\hat{d}^{m}f(0)\|_{\Theta}\right)  ^{\frac{1}{m}}=0.$

\noindent The linear subspace of $\mathcal{H}(E;F)$ of all functions $f$ of $\Theta$-bounded type is denoted by $\mathcal{H}_{\Theta b}(E;F)$.

\end{definition}

For each $\rho>0,$ the correspondence
\[
f\in\mathcal{H}_{\Theta b}(E;F)\mapsto \Vert f\Vert_{\Theta,\rho}%
=\sum_{m=0}^{\infty}\frac{\rho^{m}}{m!}\Vert\hat{d}^{m}f(0)\Vert_{\Theta
}<\infty\]
is a well defined seminorm and 
$\mathcal{H}_{\Theta b}(E;F)$ becomes a Fr\'echet space when endowed with the locally convex topology generated by these
seminorms (see, e.g, \cite[Proposition 2.3]{favaro_jatoba}). 

When $F=\mathbb{C}$ we write $\mathcal{H}_{\Theta b}(E)$ instead of $\mathcal{H}_{\Theta b}(E;\mathbb{C})$ and
when $\Theta$ is the \textit{current holomorphy type}, that is when $\mathcal{P}_{\Theta}(^{m}E;F)
 = \mathcal{P}(^{m}E;F)$ for every $m \in \mathbb{N}_{0}$, we write $\mathcal{H}_{b}(E;F)$ instead of
$\mathcal{H}_{\Theta b}(E;F)$ and $\Vert \cdot\Vert_{\rho}$ instead of $\Vert \cdot\Vert_{\Theta,\rho}$.

\medskip

Following Colombeau and Matos \cite{colombeau_matos} we introduce a vector subspace of $\mathcal{H}(E)$, when $E$ is a locally convex space, that will play a central role in this paper.
 
\begin{definition}
\label{def_holomorfas_elc}\rm Let $E$ be a locally convex space and $F$ a Banach space. A mapping 
$f \in \mathcal{H}(E;F)$ is said to be of \textit{uniform $\Theta$-bounded type} if there exist 
$p \in cs(E)$ and $f_{p} \in \mathcal{H}_{\Theta b}(E_{p};F)$ such that $f = f_{p} \circ \pi_{p}$. 
Let $\mathcal{H}_{u \Theta b}(E;F)$ denote the vector space of all holomorphic functions of uniform 
$\Theta$-bounded type from $E$ to $F$. Let $\pi_{p}^*$ denote the injective mapping
$$ \pi_{p}^{*}\colon f_{p} \in \mathcal{H}_{\Theta b}(E_{p};F) \rightarrow f_{p} \circ \pi_{p}
\in \mathcal{H}_{u \Theta b}(E;F). $$ 
Then 
$$ \mathcal{H}_{u \Theta b}(E;F) = \bigcup_{p \in cs(E)} \pi_{p}^{*}(\mathcal{H}_{\Theta b}(E_{p};F)) $$
and we endow $\mathcal{H}_{u \Theta b}(E;F)$ with the locally convex inductive topology with respect 
to the mappings $\pi_{p}^{*}$.
Thus $\mathcal{H}_{u \Theta b}(E;F)=\textrm{ind } \mathcal{H}_{\Theta b}(E_p;F)$ is an inductive limit of Fr\'echet spaces.
\end{definition}

If $D \subset cs(E)$ is a fundamental family, then it is clear that
$$ \mathcal{H}_{u \Theta b}(E;F) = \bigcup_{p \in D} \pi_{p}^{*}(\mathcal{H}_{\Theta b}(E_{p};F)). $$ When 
$\Theta$ is the current holomorphy type, then we write $\mathcal{H}_{ub}(E;F)$ instead of 
$\mathcal{H}_{u \Theta b}(E;F)$. and when $F=\mathbb{C}$ we write $\mathcal{H}_{u\Theta b}(E)$ instead of $\mathcal{H}_{u\Theta b}(E;\mathbb{C})$

\medskip

The next definition is a slight variation of the concept of $\pi_{1}$-holomorphy type (originally introduced in \cite[Definitions 2.3]{favaro_jatoba}%
) and it can be found in \cite[Definition 2.5]{BBFJ}. 

\begin{definition}\rm 
\label{pi-tipo de holomorfia} Let $E$ and $F$ be normed spaces, with $F$ complete. A holomorphy type $(\mathcal{P}_{\Theta}%
(^{m}E;F))_{m=0}^{\infty}$ from $E$ to $F$ is said to be a \emph{$\pi_{1}$-holomorphy type} if the following conditions hold:

\item[(a1)] Polynomials of finite type belong to $(\mathcal{P}_{\Theta}%
(^{m}E;F))_{m=0}^{\infty}$ and there exists $K>0$ such that $$\Vert\phi^{m}\cdot b\Vert_{\Theta
}\leq K^{m}\Vert\phi\Vert^{m}\cdot\Vert b\Vert$$ for all $\phi\in E^{\prime}_b$,
$b\in F$ and $m\in\mathbb{N}$;

\item[(a2)] For each $m\in\mathbb{N}_{0}$, $\mathcal{P}_{f}(^{m}E;F)$ is dense in $(\mathcal{P}_{\Theta}(^{m}E;F),\Vert\cdot\Vert_{\Theta})$.

\end{definition}

The main examples that we are interested in are the following:

\begin{example}\label{nuclear}\rm
Let $E$ and $F$ be normed spaces, with $F$ complete. 

\item[(a)] It is clear that the sequence of nuclear polynomials $\left(\mathcal{P}_{N}\left(^{m}E;F\right)\right)_{m=0}^\infty$ is a $\pi_1$-holomorphy type (see \cite{favaro_jatoba} or \cite{gupta}), which defines the Fr\'echet space $\mathcal{H}_{Nb}\left(E;F\right)$ of entire functions of nuclear bounded type.

\item[(b)] A polynomial $P\in\mathcal{P}\left(^{m}E;F\right)$  is said to be \emph{approximable} if $P\in\overline{\mathcal{P}_f\left(^{m}E;F\right)}^{\Vert\cdot\Vert}$. Let $\mathcal{P}_{A}\left(^{m}E;F\right)$ denotes the subspace of all approximable members of $\mathcal{P}\left(^{m}E;F\right)$, endowed with the sup norm. Then it is clear that the sequence $\left(\mathcal{P}_{A}\left(^{m}E;F\right)\right)_{m=0}^\infty$ is a $\pi_1$-holomorphy type, which defines the Fr\'echet space $\mathcal{H}_{Ab}\left(E;F\right)$ of entire functions of approximable bounded type.
\end{example}

\begin{definition}
\label{borel}\rm  Let $E$ and $F$ be normed spaces, with $F$ complete and let $(\mathcal{P}_{\Theta}(^{m}E;F))_{m=0}^{\infty}$ 
be a $\pi_{1}$-holomorphy type.

(a)  We recall that the \textit{polynomial Borel transform} 
$$ \mathcal{B}_{m}\colon \mathcal{P}_{\Theta}(^{m}E;F)^{\prime} \rightarrow \mathcal{P}(^{m}E^{\prime}_b; F^{\prime}_b) $$ 
is defined by
$$ (\mathcal{B}_{m}T)(\phi)(y) = T(\phi^{m}y) \hspace{.1in} \mbox{for every $T \in \mathcal{P}_{\Theta}(^{m}E;F)^{\prime}$, $\phi \in E^{\prime}_b$, $y \in F$}. $$ 
Then $\mathcal{B}_{m}$ is linear, continuous and injective. The image of $\mathcal{B}_{m}$ in $\mathcal{P}(^{m}E^{\prime}_b; F^{\prime}_b)$ is denoted by
$\mathcal{P}_{\Theta^{\prime}}(^{m}E^{\prime}_b; F^{\prime}_b)$, and the function 
$$ \mathcal{B}_{m}T \in \mathcal{P}_{\Theta^{\prime}}(^{m}E^{\prime}_b; F^{\prime}_b) \rightarrow \|T\| \in \mathbb{R} $$
defines a norm $\|.\|_{\Theta^{\prime}}$ on $\mathcal{P}_{\Theta^{\prime}}(^{m}E^{\prime}_b; F^{\prime}_b)$. Thus 
$(\mathcal{P}_{\Theta}(^{m}E;F)^{\prime}, \|\cdot\|)$ is isometrically isomorphic to 
$(\mathcal{P}_{\Theta^{\prime}}(^{m}E^{\prime}_b;F^{\prime}_b), \|\cdot\|_{\Theta^{\prime}})$. 

\medskip

\noindent (b) A function $f \in \mathcal{H}(E^{\prime}_b)$ is said to be of \textit{$\Theta^{\prime}$-exponential type} if 
$\hat{d}^{m}f(0) \in \mathcal{P}_{\Theta^{\prime}}(^{m}E^{\prime}_b)$ for every $m \in \mathbb{N}_{0}$, 
and there are $C, c \ge 0$ such that $\| \hat{d}^{m}f(0) \|_{\Theta^{\prime}} \le Cc^{m}$ for every $m \in \mathbb{N}_{0}$. 
The vector space of all entire functions 
of $\Theta^{\prime}$-exponential type on $E^{\prime}_b$ is denoted by $Exp_{\Theta^{\prime}}(E^{\prime}_b)$ (see \cite[p. 915]{favaro_jatoba}). 
\end{definition}

By \cite[Theorem 2.1]{favaro_jatoba} the \textit{holomorphic Borel transform}
$$ \mathcal{B}\colon \left[\mathcal{H}_{\Theta b}(E)\right]^{\prime}_b \rightarrow Exp_{\Theta^{\prime}}(E^{\prime}_b), $$ 
which is defined by
$$ (\mathcal{B}T)(\phi) = T(e^{\phi}) \hspace{.1in} \mbox{for every $T \in \left[\mathcal{H}_{\Theta b}(E)\right]^{\prime}_b$ and $\phi \in E^{\prime}_b$}, $$
is a vector space isomorphism.

\medskip

Let $E$ be a locally convex space such that there exist a fundamental family $D\subset cs(E)$ such that $(\mathcal{P}_{\Theta}(^{m}E_p))_{m=0}^{\infty}$ is a $\pi_{1}$-holomorphy type for every $p\in D$. Then it is clear that for each $\phi \in E^{\prime}_b$ there exist $p \in cs(E)$ and 
$\phi_{p} \in (E_{p})^{\prime}_b$ such that $\phi = \phi_{p} \circ \pi_{p}$.
Thus, for $T \in \mathcal{H}_{u \Theta b}(E)^{\prime}$ we have
$$ T(e^{\phi}) = T(e^{\phi_{p} \circ \pi_{p}}) = T \circ \pi_{p}^{*}(e^{\phi_{p}}) = T_{p}(e^{\phi_{p}}), $$
with $T_{p} = T \circ \pi_{p}^{*} \in \mathcal{H}_{\Theta b}(E_{p})$. By the preceding definition the function 
$\phi_{p} \in (E_{p})^{\prime}_b \rightarrow T_{p}(e^{\phi_{p}}) \in \mathbb{C}$ belongs to $Exp_{\Theta^{\prime}}((E_{p})^{\prime}_b)$. 
\vspace{.1in}

Now it makes sense the next definition:

\begin{definition}\label{def-borel} 
\rm Let $E$ be a locally convex space such that there exist a fundamental family $D\subset cs(E)$ such that $(\mathcal{P}_{\Theta}(^{m}E_p))_{m=0}^{\infty}$ is a $\pi_{1}$-holomorphy type for every $p\in D$. We denote by 
$Exp_{\Theta^{\prime}}(E^{\prime}_b)$ the subspace of all $f \in \mathcal{H}(E^{\prime}_b)$ such that
$f \circ \pi_{p}^{*} \in Exp_{\Theta^{\prime}}((E_{p})_b^{\prime})$ for some $p \in cs(E)$. We define the
\textit{holomorphic Borel transform}
$$ \mathcal{B}\colon \left[\mathcal{H}_{u \Theta b}(E)\right]^{\prime}_b \rightarrow Exp_{\Theta^{\prime}}(E^{\prime}_b) $$ by
$$ (\mathcal{B}T)(\phi) = T(e^{\phi}) \hspace{.1in} \mbox{for every $T \in \left[\mathcal{H}_{u\Theta b}(E)\right]^{\prime}_b$,
$\phi \in E^{\prime}_b$}. $$ 
\end{definition}

\medskip

Finally we recall the concept of convolution operator on $\mathcal{H}_{\Theta b}(E)$ when $E$ is a normed space and we introduce convolution operators on $\mathcal{H}_{u\Theta b}(E)$ in the case that $E$ is a locally convex space. 

\begin{definition}\label{convolucao}\rm
(\cite[Definition 3.1]{favaro_jatoba}) Let $E$ be a normed space.\\
(a) A \textit{convolution operator} on $\mathcal{H}_{\Theta b}(E)$ 
is a continuous linear mapping 
$$ L\colon \mathcal{H}_{\Theta b}(E) \rightarrow \mathcal{H}_{\Theta b}(E) $$ 
such that $L(\tau_{a}f) = \tau_{a}(Lf)$ for every $f \in \mathcal{H}_{\Theta b}(E)$ and $a \in E$ We recall that $(\tau_{a}f)(x) = f(x-a)$ for every $x \in E$ and we denote by $\mathcal{A}_{\Theta b}(E)$ the vector space of all convolution operators on $\mathcal{H}_{\Theta b}(E)$.\\
(b) The linear mapping
$$ \Gamma\colon \mathcal{A}_{\Theta b}(E) \rightarrow \mathcal{H}_{\Theta b}(E)^{\prime} $$ is defined by
$$ (\Gamma L)(f) = (Lf)(0) \hspace{.1in} \mbox{for every $L \in \mathcal{A}_{\Theta b}(E)$ and 
$f \in \mathcal{H}_{\Theta b}(E)$}. $$
\end{definition}

\begin{definition} 
\label{def-convolution}
\rm Let $E$ be a locally convex space. A \textit{convolution operator} 
on $\mathcal{H}_{u \Theta b}(E)$ is a continuous linear mapping 
$$ L\colon \mathcal{H}_{u \Theta b}(E) \rightarrow \mathcal{H}_{u \Theta b}(E) $$ 
such that $L(\tau_{a}f) = \tau_{a}(Lf)$ for every $f \in \mathcal{H}_{u \Theta b}(E)$ and $a \in E$.

We denote by $\mathcal{A}_{u \Theta b}(E)$ the vector space of all convolution operators on $\mathcal{H}_{u \Theta b}(E)$.
The linear mapping
$$ \Gamma\colon \mathcal{A}_{u \Theta b}(E) \rightarrow \mathcal{H}_{u \Theta b}(E)^{\prime} $$ 
is defined by
$$ (\Gamma L)(f) = (Lf)(0) \hspace{.1in} \mbox{for every $L \in \mathcal{A}_{u \Theta b}(E)$ and
$f \in \mathcal{H}_{u \Theta b}(E)$}. $$

\end{definition}

\section{Convolution operators on spaces of entire functions on $(DFN)$-spaces}

A $(DFN)$-space is the strong dual of a Fr\'echet nuclear space. Nuclear spaces were introduced by Grothendieck \cite{groth} and together with normed spaces are the most important classes of locally convex spaces encountered in analysis. A very good reference for the theory of nuclear spaces is the book of Pietsch \cite{pietsch}. Holomorphic functions on nuclear spaces were first investigated by Boland \cite{boland}, but many other authors have worked in that direction. We mention, among may others, \cite{Bo, bo-di, BMV}.

The main result of this section is the following theorem. 

\begin{theorem}
\label{theo_hypercyclic}
Let $E$ be a (DFN)-space, and let $L$ be a nontrivial convolution operator on 
$\mathcal{H}(E)$. Then $L$ is mixing, in particular hypercyclic.
\end{theorem}

Our proof of \ref{theo_hypercyclic} rests on the following theorem, which, as mentioned in the Introduction, is due to Costakis and Sambarino \cite{costakis} and sharpens an earlier result of Kitai \cite{kitai} and Gethner and Shapiro \cite{GS}. 

\begin{theorem}
\label{criterion}Let $X$ be a separable Fr\'echet space. Then a continuous linear 
mapping $T\colon X \rightarrow X$ is mixing if there are dense subsets $Z,Y$ of $X$ and a mapping
$S\colon Y \rightarrow Y$ satisfying the following three conditions:

\textit{(a) $T^{n}(z) \rightarrow 0$ when $n \rightarrow \infty$ for every $z \in Z$.}

\textit{(b) $S^{n}(y) \rightarrow 0$ when $n \rightarrow \infty$ for every $y \in Y$.}

\textit{(c) $T \circ S (y) = y$ for every $y \in Y$.}

\end{theorem}

Before proving Theorem \ref{theo_hypercyclic} we need some auxiliary results.

\begin{proposition}
\label{exp}
Let $E$ be a locally convex space, and 
assume there is a fundamental family $D \subset cs(E)$ such that the sequence
$(\mathcal{P}_{\Theta}(^{m}E_{p}))_{m=0}^{\infty}$ is a $\pi_{1}$-holomorphy type
for every $p \in D$. Then:

(a) The set
$$ S_{U} = {\rm span } \{e^{\phi}: \phi \in U \} $$
is a dense subspace  of $\mathcal{H}_{u \Theta b}(E)$ for each nonvoid open subset 
$U$ of $E^{\prime}_b$.

(b) The set
$$ B = \{e^{\phi}: \phi \in E^{\prime} \} $$
is a linearly independent subset of $\mathcal{H}_{u \Theta b}(E)$.

\end{proposition}

\begin{proof} 
(a) Let $U$ be a nonvoid open subset of $E^\prime$. For each $p \in D$ consider the mapping
$$\pi_p^{\prime}\colon \phi_p\in (E_p)^\prime_b\rightarrow \phi_p\circ\pi_p\in E^{\prime}_b$$
and observe that $\pi_{p}^{\prime}((E_{p})^{\prime}_b) \subset E^{\prime}_b$ and $\pi_{p}^{\prime}$
is continuous. Let $U_{p} = (\pi_{p}^{\prime})^{-1}(U)$. Then $U_{p}$ is a nonvoid open subset of $(E_{p})^{\prime}_b$. Let
$$ S_{U_{p}} = {\rm span } \{e^{\phi_{p}}: \phi_{p} \in U_{p} \}. $$
By \cite[Proposition 4.3]{BBFJ} $S_{U_{p}}$ is a dense subspace of $\mathcal{H}_{\Theta b}(E_{p})$.
Since $$\mathcal{H}_{u \Theta b}(E) = \bigcup_{p \in D} \pi_{p}^{*}(\mathcal{H}_{\Theta b}(E_{p}))$$  it follows that 
$S_{U}$ is a dense subspace of $\mathcal{H}_{u \Theta b}(E)$.

\medskip

(b) If we set
$$ B_{p} = \{ e^{\phi_{p}}: \phi_{p} \in (E_{p})^{\prime}_b \} $$ 
for every $p \in D$, then it is clear that $B = \bigcup_{p \in D} \pi_{p}^{*}(B_{p})$. 
By \cite[Proposition 4.6]{BBFJ}, each $B_{p}$ is a linearly independent subset of $\mathcal{H}_{\Theta b}(E_{p})$. 
Since each $\pi_{p}^{*}$ is injective, 
it follows that $B$ is a linearly independent subset of $\mathcal{H}_{u \Theta b}(E)$.
\end{proof}

\begin{lemma}
\label{lemma_hypercyclic} Let $E$ be a locally convex space, and assume
there is a fundamental family $D \subset cs(E)$ such that the sequence
$(\mathcal{P}_{\Theta}(^{m}E_{p}))_{m=0}^{\infty}$ is a $\pi_{1}$-holomorphy type for every $p \in D$.
Let $L$ be a convolution operator on $\mathcal{H}_{u \Theta b}(E)$. Then:

\textit{(a) $L(e^{\phi}) = \mathcal{B}(\Gamma L)(\phi) e^{\phi}$ for every $\phi \in E^{\prime}_b$.} 

\textit{(b) $L$ is a scalar multiple of the identity operator if and only if the entire function
$\mathcal{B}(\Gamma L)\colon E^{\prime}_b \rightarrow \mathbb{C}$ is constant.}
\end{lemma}

\begin{proof}
(a) If $\phi \in E^{\prime}_b$, then it follows from Definitions \ref{def-borel} and \ref{def-convolution} that
$$ \mathcal{B}(\Gamma L)(\phi) = (\Gamma L)(e^{\phi}) = L(e^{\phi})(0). $$ 
Hence for each $y \in E$ it follows that
$$ [\mathcal{B} (\Gamma L) (\phi) e^{\phi}](y)  = \mathcal{B} (\Gamma L) (\phi) e^{\phi(y)} = e^{\phi(y)} (\Gamma L) (e^{\phi})  $$
$$ = e^{\phi(y)} (L e^{\phi})(0) = [L(e^{\phi(y)} e^{\phi})](0) = [L(\tau_{-y}e^{\phi})](0) $$
$$ = [\tau_{-y}(Le^{\phi})](0) = L(e^{\phi})(y). $$

(b) Let $\lambda \in \mathbb{C}$ such that $\mathcal{B}(\Gamma L)(\phi) = \lambda$ for every $\phi \in E^{\prime}_b$.
It follows from (a) that
$$ L e^{\phi} = \mathcal{B} (\Gamma L) (\phi) e^{\phi} = \lambda e^{\phi} \hspace{.1in} \mbox{for every $\phi \in E^{\prime}_b$}. $$ 
Since $ {\rm span \ } \{e^{\phi}: \phi \in E^{\prime} \}$ is dense in $\mathcal{H}_{u \Theta b}(E)$, it follows that
$Lf = \lambda f$ for every $f \in \mathcal{H}_{u \Theta b}(E)$.

Conversely let $\lambda \in \mathbb{C}$ such that $Lf = f$ for every $f \in \mathcal{H}_{u \Theta b}(E)$.
It follows from (a) that
$$ \lambda e^{\phi} = L e^{\phi} = \mathcal{B} (\Gamma L) (\phi) e^{\phi}, $$ 
and therefore $\mathcal{B} (\Gamma L) (\phi) = \lambda$ for every $\phi \in E^{\prime}_b$.
\end{proof}

\begin{proposition}
\label{prop_hypercyclicity}
Let $E$ be a locally convex space and assume 
there is a fundamental family $D \subset cs(E)$ such that the sequence
$(\mathcal{P}_{\Theta}(^{m}E_{p}))_{m=0}^{\infty}$ is a $\pi_{1}$-holomorphy type for every $p \in D$.
Let $L$ be a nontrivial convolution operator on $\mathcal{H}_{u \Theta b}(E)$. Consider the sets 
$$ V = \{\phi \in E^{\prime}_b: |\mathcal{B}(\Gamma L)(\phi)| < 1 \} = [\mathcal{B}(\Gamma L)]^{-1}(\Delta) $$ and
$$ W = \{\phi \in E^{\prime}_b: |\mathcal{B}(\Gamma L)(\phi)| > 1 \} = [\mathcal{B}(\Gamma L)]^{-1}(\mathbb{C} \setminus \overline{\Delta}). $$ 
Consider also the sets
$$ H_{V} = {\rm span  } \{e^{\phi}: \phi \in V \} \hspace{.1in} and \hspace{.1in} H_{W} = {\rm span  } \{e^{\phi}: \phi \in W \}. $$ Then:

\textit{(a) $H_{V}$ and $H_{W}$ are dense subspaces of $\mathcal{H}_{u \Theta b}(E)$.} 

\textit{(b) $L^{n}f \rightarrow 0$ when $n \rightarrow \infty$ for each $f \in H_{V}$.} 

\textit{(c) If we define
$$ S(e^{\phi}) = \frac{ e^{\phi} }{ \mathcal{B}(\Gamma L)(\phi) } \hspace{.1in} \mbox{ for every $\phi \in W$ }, $$ 
then $S$ admits a unique extension to a linear mapping $S\colon H_{W} \rightarrow H_{W}$, and 
$S^{n}f \rightarrow 0$ when $n \rightarrow \infty$ for each $f \in H_{W}$.}

\textit{(d) $L \circ S (f) = f$ for every $f \in H_{W}$.} 
\end{proposition}

\begin{proof}
(a) Since $L$ is not a scalar multiple of the identity, Lemma \ref{lemma_hypercyclic}(b) implies that the entire function
$\mathcal{B}(\Gamma L)\colon E^{\prime}_b \rightarrow \mathbb{C}$ is not constant. Hence $V$ and $W$ are nonvoid open subsets 
of $E^{\prime}_b$. By Proposition \ref{exp}(a) $H_{V}$ and $H_{W}$ are dense subspaces of $\mathcal{H}_{u \Theta b}(E)$. 

(b) Given $\phi \in V$, Lemma \ref{lemma_hypercyclic}(a) implies that
$$ L(e^{\phi}) = \mathcal{B}(\Gamma L)(\phi)e^{\phi} \in H_{V}. $$ 
Since $L$ is linear, it is clear that $L(H_{V}) \subset H_{V}$. It is easy to see that
$$ L^{n}(e^{\phi}) = [\mathcal{B}(\Gamma L)(\phi)]^{n}e^{\phi} \hspace{.1in} \mbox{for every $\phi \in V$, $n \in \mathbb{N}$}. $$ 
Now let $f \in H_{V}$, that is $f = \sum_{j=1}^{m} \alpha_{j} e^{\phi_{j}}$, with $\alpha_{j} \in \mathbb{C}$ and 
$\phi_{j} \in V$. It follows that 
$$ L_{n}(f) = \sum_{j=1}^{m} \alpha_j L^{n}(e^{\phi_{j}}) = \sum_{j=1}^{m} \alpha_{j} [\mathcal{B}(\Gamma L)(\phi_{j})]^{n}e^{\phi_{j}}. $$ 
Since $|\mathcal{B}(\Gamma L)(\phi_{j})|<1$ for every $j=1,\ldots,m$, it follows that $L^{n}f \rightarrow 0$ when $n \rightarrow \infty$. 

(c) If $\phi \in W$, then $\mathcal{B}(\Gamma L)(\phi) \ne 0$. Hence we may define
$$ S(e^{\phi}) = \frac{e^{\phi}}{\mathcal{B}(\Gamma L)(\phi)} \in H_{W}. $$          
It is easy to see that
$$ S^{n}(e^{\phi}) = \frac{e^{\phi}}{[\mathcal{B}(\Gamma L)(\phi)]^{n}} \hspace{.1in} \mbox{for every $\phi \in W$, $n \in \mathbb{N}$}. $$     
By Proposition \ref{exp}(b) $\{e^{\phi}: \phi \in W \}$ is a Hamel basis of $H_{W}$. Hence $S$ admits a unique extension
to a linear mapping $S\colon H_{W} \rightarrow H_{W}$. Now let $f \in H_{W}$, that is $f = \sum_{j=1}^{m} \alpha_{j} e^{\phi_{j}}$,
with $\alpha_{j} \in \mathbb{C}$ and $\phi_{j} \in W$. It follows that
$$ S^{n}f = \sum_{j=1}^{m} \frac{ \alpha_{j} e^{\phi_{j}} }{[\mathcal{B}(\Gamma L)(\phi_{j})]^{n}}. $$
Since $|\mathcal{B}(\Gamma L))(\phi_{j})| > 1$ for every $j$, it follows that $S^{n}f \rightarrow 0$ when $n \rightarrow \infty$.

(d) It is clear that $L \circ S (e^{\phi}) = e^{\phi}$ for every $\phi \in W$, and therefore $L \circ S (f) = f$ for every 
$f \in H_{W}$.
\end{proof}

\textit{Proof of Theorem \ref{theo_hypercyclic}.} By \cite[Theorem 6.5]{colombeau_mujica}, $\mathcal{H}(E) = \mathcal{H}_{u N b}(E)$ algebraically and topologically. By a result of Boland \cite[Corollary 1.4]{Bo}, $\mathcal{H}(E)$ is a Fr\'echet nuclear space.
In particular $\mathcal{H}_{u N b}(E)$ is a separable Fr\'echet space. 
If $D \subset cs(E)$ is any fundamental family, then the sequence $(\mathcal{P}_{N}(^{m}E_{p}))_{m=0}^{\infty}$ is a $\pi_{1}$-holomorphy type for every $p \in D$, by Example \ref{nuclear}.
By Proposition \ref{prop_hypercyclicity} $H_{V}$ and $H_{W}$ are dense subspaces of $\mathcal{H}_{uNb}(E)$, and there is a linear mapping
$S\colon H_{W} \rightarrow H_{W}$ such that

(a) $L^{n}f \rightarrow 0$ when $n \rightarrow \infty$ for every $f \in H_{V}$;

(b) $S^{n}f \rightarrow 0$ when $n \rightarrow \infty$ for every $f \in H_{W}$;

(c) $L \circ S (f) = f$ for every $f \in H_{W}$.

By Theorem \ref{criterion} the operator $L$ is mixing.   \hfill$\square$

\section{Convolution operators on spaces of entire functions on $(DFC)$-spaces}

A $(DFC)$-space is a locally convex space of the form $E=F_c^{\prime}$, where $F$ is a Fr\'echet space. $(DFC)$-spaces were first studied by Brauner \cite{brauner} and H\"ollstein \cite{holl1,holl2}. Holomorphic functions on $(DFC)$-spaces have been studied by Mujica \cite{mujica_dfc}, Valdivia \cite{valdivia}, Schottenloher \cite{schot}, Nachbin \cite{nachbin4}, Louren\c co \cite{lilian} and Galindo et al. \cite{GGM}.

The main result in this section is the following theorem.

\begin{theorem} 
\label{theo_dfc}
Let $E = F^{\prime}_{c}$, where $F$ is a separable Fr\'echet space with the approximation property.
Let $L$ be a nontrivial convolution operator on $(\mathcal{H}(E), \tau_{0})$. Then $L$ is mixing, in particular hypercyclic.
\end{theorem}

The proof of Theorem \ref{theo_dfc} rests on Theorem \ref{criterion}, but before proving the theorem we need some auxiliary results.

\begin{proposition}
Let $E = F^{\prime}_{c}$, where $F$ is a Fr\'echet space. Then:

\textit{(a) $E$ is a semi-Montel, hemicompact k-space.}

\textit{(b) $(\mathcal{H}(E), \tau_{0})$ is a Fr\'echet space.} \vspace{.1in}
\end{proposition}

\begin{proof}
 (a) By \cite[Proposition 7.2]{mujica_dfc} $E$ is a semi-Montel, hemicompact space. By the Banach-Dieudonn\'e theorem  
(see \cite[p. 245, Theorem 1]{horvath}) $E$ is a $k$-space. (b) follows at once from (a).
\end{proof}

If $E = F^{\prime}_{c}$, where $F$ is a Fr\'echet space, then a result of Schwartz guarantees that $F$ has the approximation property
if and only if $E$ has the approximation property (see \cite[Expos\'e $n^{0}$ 14, Th\'eor\`eme 2]{schwartz} or \cite[Corollary 1.3]{DiMu}).

\begin{proposition}
\label{HE_frechet}
 \textit{Let $E = F^{\prime}_{c}$, where $F$ is a separable Fr\'echet space with the approximation property.
Then $(\mathcal{H}(E), \tau_{0})$ is a separable Fr\'echet space.}
\end{proposition}

\begin{proof}
By considering the Taylor series we see that every $f \in \mathcal{H}(E)$ can be approximated, uniformly on compact sets, 
by continuous polynomials on $E$. Since $E$ has the approximation property, it is clear that every continuous polynomial on $E$ can be approximated, 
uniformly on compact sets, by continuous polynomials of finite type. By the Mackey-Arens theorem (see \cite[p. 205, Theorem 1]{horvath}) $E^{\prime}_{b} = E^{\prime}_{c} = F$ is separable, it follows that
$(\mathcal{P}_{f}(^{m}E), \tau_{0})$ is separable for every $m \in \mathbb{N}_{0}$. Hence it follows that $(\mathcal{H}(E), \tau_{0})$
is separable, as asserted.
\end{proof}

\begin{definition}\rm
Let $E$ and $F$ be normed spaces. An operator $T \in \mathcal{L}(E;F)$ is said to be \textit{approximable} if
$T \in \overline{E^{\prime} \otimes F}$.
\end{definition}

\begin{lemma} \textit{Let $E$, $F$ and $G$ be normed spaces, with $G$ complete, and let $T \in \mathcal{L}(F;E)$ be an approximable operator. Then 
$f \circ T \in \mathcal{H}_{Ab}(F;G)$ for every $f \in \mathcal{H}_{b}(E;G)$, and the mapping
$$ f \in \mathcal{H}_{b}(E;G) \rightarrow f \circ T \in \mathcal{H}_{Ab}(F;G) $$ is linear and continuous.}
\end{lemma}

\begin{proof}
Let $\sum_{m=0}^{\infty} P_{m}(x)$ denote the Taylor series of $f$ at the origin. Then
$$ f \circ T(y) = \sum_{m=0}^{\infty} P_{m} \circ T(y) \hspace{.1in} \mbox{for every $y \in F$}. $$
Since $T$ is approximable, there is a sequence $(T_{n})_{n=1}^{\infty} \in F^{\prime} \otimes E$
such that $\|T-T_{n}\| \rightarrow 0$. Since $P_{m} \in \mathcal{P}(^{m}E)$ for every $m \in \mathbb{N}_{0}$, it is clear that
$P_{m} \circ T \in \mathcal{P}(^{m}F)$ and $\|P_{m} \circ T\| \le \|P_{m}\| \|T\|^{m}$ for every $m \in \mathbb{N}_{0}$. 

Since $f \in \mathcal{H}_{b}(E;G)$, the Taylor series of $f$ at the origin has an infinite radius of convergence.
By the Cauchy-Hadamard formula (see \cite[Theorem 4.3]{mujica}),  $\|P_{m}\|^{1/m} \rightarrow 0$. Hence it follows that
$$ \|P_{m} \circ T\|^{1/m} \leq \|P_{m}\|^{1/m} \|T\| \rightarrow 0. $$
Hence the Taylor series of $f \circ T$ at the origin has also an infinite radius of convergence, and therefore
$f \circ T \in \mathcal{H}_{b}(F;G)$. 

To show that $f \circ T \in \mathcal{H}_{Ab}(F;G)$ we have to prove that $P_{m} \circ T \in \mathcal{P}_{A}(F;G)$ for every 
$m \in \mathbb{N}_{0}$. By the Newton binomial formula (see \cite[Corollary 1.9]{mujica}), for every $y \in F$ with $\|y\| \le 1$ it follows that
$$ \Vert P_{m} \circ T(y)-P_{m} \circ T_{n}(y)\Vert = \left\| \sum_{k=1}^{m}
\binom{m}{k} \check{P}_{m} (T_{n}y)^{m-k}(Ty-T_{n}y)^{k} \right\| $$
$$ \le \sum_{k=1}^{m} \binom{m}{k}  \|\check{P}_{m}\| \|T_{n}\|^{m-k}\|T-T_{n}\|^{k} $$
$$ \le e^{m} \|P_{m}\| \|T-T_{n}\| \sum_{k=1}^{m} \binom{m}{k}  c^{m-k} d^{k-1} $$ 
 for suitable positive constants $c$ and $d$. Therefore $\| P_{m} \circ T - P_{m} \circ T_{n} \| \rightarrow 0$.
Since $T_{n} \in F^{\prime} \otimes E$, it is clear that
$P_{m} \circ T_{n} \in \mathcal{P}_{f}(^{m}F)$, and therefore $P_{m} \circ T \in \mathcal{P}_{A}(^{m}F;G)$.

Finally it is clear that the mapping $f \in \mathcal{H}_{b}(E;G) \rightarrow f \circ T \in \mathcal{H}_{Ab}(F;G)$ is continuous, since
$$ \|f \circ T\|_{\rho} = \sum_{m=0}^{\infty} \|P_{m} \circ T\|\rho^{m} \le \sum_{m=0}^{\infty} \|P_{m}\| \|T\|^{m}\rho^{m} = \|f\|_{\|T\|\rho} $$
for every $\rho>0$.
\end{proof}

\begin{theorem}
\label{theo_Huab}
 \textit{Let $E$ be a (DFC)-space with the approximation property. Then
$\mathcal{H}_{uAb}(E) = \mathcal{H}_{ub}(E) = (\mathcal{H}(E), \tau_{0})$ algebraically and topologically.}
\end{theorem}

\begin{proof}
 We first establish the continuous inclusions
$$ \mathcal{H}_{uAb}(E) \hookrightarrow \mathcal{H}_{ub}(E) \hookrightarrow (\mathcal{H}(E), \tau_{0}). $$
Since
$$ \mathcal{H}_{uAb}(E) = \bigcup_{p \in cs(E)} \pi_{p}^{*}(\mathcal{H}_{Ab}(E_{p})) $$ and
$$ \mathcal{H}_{ub}(E) = \bigcup_{p \in cs(E)} \pi_{p}^{*}(\mathcal{H}_{b}(E_{p})), $$ it is clear that
$\mathcal{H}_{uAb}(E) \subset \mathcal{H}_{ub}(E)$, and the inclusion mapping is continuous.
Since the inclusion mapping
$\pi_{p}^{*}(\mathcal{H}_{b}(E_{p})) \hookrightarrow (\mathcal{H}(E), \tau_{0})$ 
is clearly continuous, we obtain the continuous inclusion
$\mathcal{H}_{ub}(E) \hookrightarrow (\mathcal{H}(E), \tau_{0})$. 

We next show that $\mathcal{H}_{uAb}(E) = (\mathcal{H}(E), \tau_{0})$ algebraically and topologically.
We know that $\mathcal{H}_{uAb}(E)$ is bornological, and $(\mathcal{H}(E), \tau_{0})$ is a Fr\'echet space, 
in particular bornological. Hence it suffices to show that each bounded subset of $(\mathcal{H}(E), \tau_{0})$
is contained and bounded in $\mathcal{H}_{uAb}(E)$. Let $\{f_{i}: i \in I\}$ be a bounded subset of 
$(\mathcal{H}(E), \tau_{0})$. Let $f \in \mathcal{H}(E; \ell_{\infty}(I))$ be defined by
$f(x) = (f_{i}(x))_{i \in I}$ for every $x \in E$.  By a result of Galindo et al. (see \cite[Corollary 1]{GGM}),
$$ (\mathcal{H}(E; \ell_{\infty}(I)), \tau_{0}) = \mathcal{H}_{ub}(E; \ell_{\infty}(I)) $$ 
algebraically and topologically. In particular $ f \in \mathcal{H}_{ub}(E; \ell_{\infty}(I))$
and therefore there are $p \in cs(E)$ and $f_{p} \in \mathcal{H}_{b}(E; \ell_{\infty}(I))$ such that 
$f = f_{p} \circ \pi_{p}$. By a result of Louren\c{c}o (see \cite[Lemma 2.2]{lilian}), there are $q \in cs(E)$, $q \ge p$
such that the canonical mapping $\pi_{pq}: E_{q} \rightarrow E_{p}$ is an approximate operator.
Let $f_{q} = f_{p} \circ \pi_{pq}$. By the preceding lemma $f_{q} \in \mathcal{H}_{Ab}(E_{q}; \ell_{\infty}(I))$
and
$$ f = f_{p} \circ \pi_{p} = f_{p} \circ \pi_{pq} \circ \pi_{q} = f_{q} \circ \pi_{q}. $$
Thus $f = (f_{i})_{i \in I} \in \mathcal{H}_{uAb}(E; \ell_{\infty}(I))$, and therefore $\{f_{i}: i \in I\}$
is a bounded subset of $\mathcal{H}_{uAb}(E)$, as asserted.
\end{proof}

\textit{Proof of Theorem \ref{theo_dfc}}. By Proposition \ref{HE_frechet} and Theorem \ref{theo_Huab} $\mathcal{H}_{uAb}(E) = (\mathcal{H}(E), \tau_{0})$
is a separable Fr\'echet space. If $D \subset cs(E)$ is any fundamental family, then it follows from Example \ref{nuclear} that $(\mathcal{P}_{A}(^{m}E_{p}))_{m=0}^{\infty}$
is a $\pi_{1}$-holomorphy type for every $p \in D$. By Proposition \ref{prop_hypercyclicity} there are dense subspaces $H_{V}$ and $H_{W}$ of $\mathcal{H}_{uAb}(E)$ and there is a linear mapping
$S\colon H_{W} \rightarrow H_{V}$ such that:

(a) $L^{n}f \rightarrow 0$ when $n \rightarrow \infty$ for every $f \in H_{V}$;

(b) $S^{n}f \rightarrow 0$ when $n \rightarrow \infty$ for every $f \in H_{W}$;

(c) $L \circ S (f) =f$ for every $f \in H_{W}$.

By Theorem \ref{criterion} the operator $L$ is mixing.     \hfill$\square$

\section{A counterexample}

We present a simple example of convolution operator which is not hypercyclic.
So far we have only considered the compact-open topology $\tau_{0}$ on the space $\mathcal{H}(E)$. But now we will also 
consider the compact-ported topology $\tau_{\omega}$ introduced by Nachbin \cite{nachbin1}, and the bornological topology $\tau_{\delta}$ introduced 
by Coeur\'e \cite{coeure} and Nachbin \cite{nachbin3}. For background information on these topologies we refer the reader to the book of Dineen \cite{Di2}.

\begin{theorem}

(a) $(\mathcal{H}(\mathbb{C}^{\mathbb{N}}), \tau_{0}) = (\mathcal{H}(\mathbb{C}^{\mathbb{N}}), \tau_{\omega}) \ne 
(\mathcal{H}(\mathbb{C}^{\mathbb{N}}), \tau_{\delta}) = \mathcal{H}_{ub}(\mathbb{C}^{\mathbb{N}})$.

(b) For each $a \in \mathbb{C}^{\mathbb{N}}$, the translation operator $\tau_{a}$ is a 
convolution operator on $(\mathcal{H}(\mathbb{C}^{\mathbb{N}}), \tau)$ which is not hypercyclic, for $\tau = \tau_{0}$, 
$\tau_{\omega}$ and $\tau_{\delta}$.  

\end{theorem}

\begin{proof} (a) By a result of Barroso (see \cite[p. 537, Teorema 2.2]{Ba}), $\tau_{0}=\tau_{\omega}$ on $\mathcal{H}(\mathbb{C}^{\mathbb{N}})$.
By a result of Dineen (see \cite[p. 45, Corollary 3.2]{Di1} or \cite[Example 3.24(a)]{Di2}) we have $\tau_{\omega} \ne \tau_{\delta}$ on
$\mathcal{H}(\mathbb{C}^{\mathbb{N}})$. 

To see the connection with $\mathcal{H}_{ub}(\mathbb{C}^{\mathbb{N}})$, for each $n\in\mathbb{N}$ and for $\tau = \tau_{0}$, $\tau_{\omega}$ and $\tau_{\delta}$,
consider the canonical inclusion $J_{n}\colon \mathbb{C}^{n} \hookrightarrow \mathbb{C}^{\mathbb{N}}$, 
the canonical projection $\pi_{n}\colon \mathbb{C}^{\mathbb{N}} \rightarrow \mathbb{C}^{n}$ and the 
corresponding mappings
$$ J_{n}^{*}\colon f \in (\mathcal{H}(\mathbb{C}^{\mathbb{N}}), \tau) \rightarrow f \circ J_{n} \in (\mathcal{H}(\mathbb{C}^{n}), \tau) $$ and
$$ \pi_{n}^{*}\colon f_{n} \in (\mathcal{H}(\mathbb{C}^{n}), \tau) \rightarrow f_{n} \circ \pi_{n} \in (\mathcal{H}(\mathbb{C}^{\mathbb{N}}), \tau). $$                 
Since $J_{n}^{*} \circ \pi_{n}^{*} $ is the identity on $\mathcal{H}(\mathbb{C}^{n})$, it follows that $(\mathcal{H}(\mathbb{C}^{n}), \tau)$ is topologically isomorphic to a complemented subspace of 
$(\mathcal{H}(\mathbb{C}^{\mathbb{N}}), \tau)$. In particular 
$\pi_n^*(\mathcal{H}(\mathbb{C}^{n}))$ is a proper closed subspace of ($\mathcal{H}(\mathbb{C}^{\mathbb{N}}), \tau)$. By a result of Barroso (see \cite[p. 38, Corol\'ario]{Ba} or \cite[Proposition 1.1]{A}) 
$$ \mathcal{H}(\mathbb{C}^{\mathbb{N}}) = \bigcup_{n=1}^{\infty} \{ f_{n} \circ \pi_{n}\colon f_{n} \in \mathcal{H}(\mathbb{C}^{n}) \}
= \bigcup_{n=1}^{\infty} \pi_{n}^{*}(\mathcal{H}(\mathbb{C}^{n})). $$                                                                
If we define $p_{n} \in cs(\mathbb{C}^{\mathbb{N}})$ by
$$ p_{n}(x) = \sup_{j \le n} |\xi_{j}| \hspace{.1in} \mbox{for every $x = (\xi_{j})_{j=1}^{\infty} \in \mathbb{C}^{\mathbb{N}}$}, $$ 
then we can readily see that the normed space $(\mathbb{C}^{\mathbb{N}})_{p_{n}}$ is topologically isomorphic to $\mathbb{C}^{n}$,
and therefore
$$ \mathcal{H}_{ub}(\mathbb{C}^{\mathbb{N}}) = ind (\mathcal{H}((\mathbb{C}^{\mathbb{N}})_{p_{n}}, \tau_{0})
= ind (\mathcal{H}(\mathbb{C}^{n}), \tau_{0}). $$
By a result of Ansemil (see \cite[Proposition 1.3]{A})
$$ (\mathcal{H}(\mathbb{C}^{\mathbb{N}}), \tau_{\delta}) = ind (\mathcal{H}(\mathbb{C}^{n}), \tau_{0}) = \mathcal{H}_{ub}(\mathbb{C}^{\mathbb{N}}). $$

(b) Let $a \in  \mathbb{C}^{\mathbb{N}}$ and assume that the translation operator 
$$ \tau_{a}\colon (\mathcal{H}(\mathbb{C}^{\mathbb{N}}), \tau_{0}) \rightarrow (\mathcal{H}(\mathbb{C}^{\mathbb{N}}), \tau_{0}) $$
were hypercyclic. Then there would exist $f \in \mathcal{H}(\mathbb{C}^{\mathbb{N}})$ such that the set
$\{f, \tau_{a}f, \tau_{a}^{2}f, ... \}$
would be dense in $(\mathcal{H}(\mathbb{C}^{\mathbb{N}}), \tau)$. Let $n\in\mathbb{N}$ be such that $f = f_{n} \circ \pi_{n}$, with $f_{n} \in \mathcal{H}(\mathbb{C}^{n})$.  
Then
$$ \{f, \tau_{a}f, \tau_{a}^{2}f,\ldots\}  = \pi_{n}^{*}(\{f_{n}, \tau_{\pi_n(a)}f_{n}, \tau_{\pi_n(a)}^{2}f_{n},\ldots\} )
\subset \pi_{k}^{*}(\mathcal{H}(\mathbb{C}^{n})) , $$
a contradiction, since $\pi_{n}^{*}(\mathcal{H}(\mathbb{C}^{n}))$ is a proper closed subspace of $(\mathcal{H}(\mathbb{C}^{\mathbb{N}}), \tau)$.

\end{proof}

\begin{corollary}
For each $a \in \mathbb{C}^{\mathbb{N}}$, the translation operator $\tau_{a}$ is a
convolution operator on $\mathcal{H}_{u b} (\mathbb{C}^{\mathbb{N}})$ which is not hypercyclic.
\end{corollary}

\bigskip

\noindent Vin\'icius V. F\'avaro\\
FAMAT-UFU\\
Av. Jo\~ao Naves de \'Avila, 2121\\
38.400-902, Uberl\^{a}ndia, Brazil\\
e-mail: vvfavaro@gmail.com

\bigskip

\noindent Jorge Mujica\\
IMECC-UNICAMP\\
Rua Sergio Buarque de Holanda, 651\\
13083-859, Campinas, SP, Brazil\\
e-mail: mujica@ime.unicamp.br

\end{document}